\documentclass{amsart}
\usepackage[demo]{graphicx}
\usepackage{caption}
\usepackage{subcaption}
\usepackage{graphicx, amsmath, amsfonts, epsfig, latexsym, amssymb, amsthm}
\usepackage{tikz}
\usepackage{tkz-berge}
\usepackage{float}

\usetikzlibrary {positioning}
\definecolor {processblue}{cmyk}{0.96,0,0,0}
\newtheorem{thm}{Theorem}[section]
\theoremstyle{definition}
\newtheorem{cor}[thm]{Corollary}
\newtheorem{prop}[thm]{Proposition}
\newtheorem{defn}[thm]{Definition}
\newtheorem{lem}[thm]{Lemma}

\newtheorem{rem}[thm]{Remark}

\numberwithin{equation}{section}
\begin{document}
\title[$z^\circ$-submodules of a reduced multiplication module]
{$z^\circ$-submodules of a reduced multiplication module}
\author{ F. Farshadifar}
\address{Department of Mathematics Education, Farhangian University, P.O. Box 14665-889, Tehran, Iran.}
\email{f.farshadifar@cfu.ac.ir}


\begin{abstract}
Let $R$ be a commutative ring with identity and $M$ be an $R$-module.
A proper  ideal $I$ of $R$ is said to be a $z^\circ$-ideal if for each $a \in I$ the intersection of all minimal prime ideals
containing $a$ is contained in $I$.
The purpose of this paper is to introduce the notion of $z^\circ$-submodules of $M$ as an extension of $z^\circ$-ideals of $R$. Moreover, we investigate some properties of this class of submodules when $M$ is a reduced multiplication $R$-module.
\end{abstract}
\maketitle

\section{Introduction}
\noindent
Throughout this paper, $R$ is a commutative ring with
identity.
A proper ideal $I$ of $R$ is called a \textit{$z$-ideal} whenever any two
elements of $R$ are contained in the same set of maximal ideals and $I$ contains
one of them, then it also contains the other one \cite{MR321915}.
The concept of a $z$-ideal, as used in the study of
rings of continuous functions, is capable of a purely
algebraic definition.
For each $a \in R$, let $\mathfrak{P}_a$, be the intersection of all minimal prime ideals of $R$ containing $a$.
A proper ideal $I$ of $R$ is called a $z^\circ$-ideal if for each $a \in I$
we have $\mathfrak{P}_a \subseteq I$ \cite{MR1736781, MR577573}.
In fact, the $z^\circ$-ideals are special types of $z$-ideals that have received a great deal of attention both in rings and
Riesz spaces.

For a submodule $N$ of an $R$-module $M$, let $\mathcal{M}(N)$ be the set of maximal submodules of $M$ containing
$N$ and $Max(M)$ be the set of all maximal submodules of $M$. The intersection of all maximal submodules of $M$ containing $N$ is said to be the \textit{Jacobson radical} of $N$ and denote by $Rad_N(M)$ \cite{MR3677368}. In case $N$ does not contained in
any maximal submodule, the Jacobson radical of $N$ is defined to be $M$. We denote the Jacobson radical of zero submodule of $M$ by $Rad_M(M)$. A proper submodule $N$ of $M$ is said to be a \textit{$z$-submodule} if for every $x,y \in M$, $\mathcal{M}(x)=\mathcal{M}(y)\not= \emptyset$ and $x \in N$ imply $y \in N$ \cite{F401}.
 A proper submodule $N$ of $M$ is said to be a \textit{strongly $z$-submodule} if $Rad_K(M) \subseteq N$  for all submodules $K$ of $N$ with $Rad_K(M)\not=M$ \cite{F401}.

Let $M$ be an $R$-module. A proper submodule $P$ of $M$ is said to be \textit{prime} if for any
$r \in R$ and $m \in M$ with $rm \in P$, we have $m \in P$ or $r \in (P :_R M)$ \cite{MR183747, MR498715}. In this case, $(P :_R M)$ is a prime ideal of $R$.
The \textit{prime spectrum} of $M$ denoted by $Spec(M)$ is the set of all prime submodules of $M$.
A prime submodule $P$ of $M$ is a minimal prime submodule over a submodule $N$ of $M$ if $P$ is a minimal element of the set of all
prime submodules of $M$ that contain $N$. \textit{A minimal prime submodule} of $M$ means a minimal prime submodule over the $0$ submodule of $M$.
The set of all minimal prime submodules of $M$ will be denoted by $Min^p(M)$. The intersection of all minimal prime submodules of $M$ containing a submodule $K$ of $M$ is denote by $\mathfrak{P}_K$. In case $K$ does not contained in
any minimal prime submodule of $M$, $\mathfrak{P}_K$ is defined to be $M$.  Also, the intersection of all minimal prime submodules of $M$ containing $x \in M$ is denote by $\mathfrak{P}_x$. In case $x$ does not contained in
any minimal prime submodule of $M$, $\mathfrak{P}_x$ is defined to be $M$. If $x\in M$ and $N$ is a submodule of $M$, define
$V (x) = \{P \in Min^p(M) : x\in P\}$ and $V (N) = \{P \in Min^p(M) : N \subseteq P\}$.

An $R$-module $M$ is said to be a \emph{multiplication module} if for every submodule $N$ of $M$ there exists an ideal $I$ of $R$ such that $N=IM$ \cite{Ba81}. Equivalently,
$M$ is a multiplication module if and only if for each submodule $N$ of $M$, we have
$N = (N :_R M)M$. An $R$-module $M$ is said to be \textit{reduced} if the intersection of all prime submodules of $M$
is equal to zero \cite{MR2839935}.

In Section 2 of this paper, we investigate some properties of reduced multiplication $R$-modules. In Section 3,
we define the notion of $z^\circ$-submodules of an $R$-module $M$ as a generalization of $z^\circ$-ideals and when $M$ is a reduced multiplication $R$-module obtaine some related results.
Some of the results in this paper are similar to those in \cite{MR194880} and \cite{MR1736781}.
\section{Reduced multiplication modules}
Let $N$ and $K$ be two submodules of an $R$-module $M$. The \emph{product} of $N$ and $K$ is defined by $(N:_RM)(K:_RM)M$  and denoted by $NK$ \cite{MR2335105}. Also, $N$ is said to be \textit{nilpotent} if $N^t = 0$ for some positive integer $t$, where $N^t$
means the product of $N$, $t$ times \cite{MR2335105}.

An element $m$ of a multiplication $R$-module $M$ is said to be \textit{nilpotent} if $(Rm)^k = 0$ for some positive integer $k$.
The set of all nilpotent elements of $M$ is denoted by $\mathfrak{N}_M$ \cite{MR1981026}.

\begin{cor}\label{cc1.3} \cite[Corollary 3.14]{MR1981026}
Let $M$ be a multiplication $R$-module. Then  $\mathfrak{N}_M$ is the intersection
of all prime submodules of $M$.
 \end{cor}

\begin{prop}\label{p991.11}
Let $M$ be a multiplication $R$-module. Then $\bigcap _{P \in Min^p(M)}P=\bigcap _{P \in Spec (M)}P$.
\end{prop}
\begin{proof}
Clearly, $\bigcap _{P \in Spec (M)}P \subseteq \bigcap _{P \in Min^p(M)}P$. For reverse inclusion, assume contrary that there exists $x \in \bigcap_{P \in Min^p(M)}P$ buy $x \not \in \bigcap _{P \in Spec (M)}P$. Thus $x \not \in P'$ for some $P' \in Spec (M)$. By \cite[Page 765]{MR932633}, there exists $P'' \in Min^p(M)$ such that $P'' \subseteq P'$. Thus $x \not \in P''$, which is a contradiction.
\end{proof}

\begin{cor}\label{C1}
 A multiplication $R$-module $M$ is reduced if and only if $\mathfrak{N}_M=0= \mathfrak{P}_0$.
\end{cor}
\begin{proof}
This follows from Corollary \ref{cc1.3} and Proposition \ref{p991.11}.
\end{proof}

Let $P$ be a prime submodule of an $R$-module $M$. Set
$O_P = \{x\in M : Ann(x) \not \subseteq  (P:_RM)\}$.
Then $O_P\subseteq P$ \cite{MR2839935}.
\begin{lem}\label{ll1.3} \cite[Corollary 3.3]{MR2839935}
Let $M$ be a reduced multiplication $R$-module. Then $P$ is a minimal prime submodule of $M$
if and only if $P = O_P$.
 \end{lem}

\begin{lem}\label{ll1.4}
Let $M$ be a reduced multiplication $R$-module and $P$ be a minimal prime submodule of $M$.
 If $x\in  P$, then there exists $a \in Ann_R(x) \setminus (P:_RM)$.
\end{lem}
\begin{proof}
This follows from Lemma \ref{ll1.3}.
\end{proof}

\begin{rem}\label{rr1.11}
Let $M$ be a multiplication $R$-module and $\Omega$ be a subset of $Min^p(M)$. Set $\mathfrak{P}_{\Omega}=\cap \{P: P \in \Omega\}$.
A subset $\Omega $ of $Min^p(M)$ is said to be closed if $\Omega=V(\mathfrak{P}_{\Omega})$. With this notion of
closed set, one can see that the space of minimal prime submodules of $M$ becomes a topological space.
\end{rem}

\begin{thm}\label{trt1.11}
Let $M$ be a reduced multiplication $R$-module. Then for each $x \in M$,  we have $V(Ann_R(x)M)=Min^p(M) \setminus V(x)$. In particular, $V(Ann_R(x)M)$ and $V(x)$ are disjoint open-and-closed sets.
\end{thm}
\begin{proof}
If $P\in V(x)$, then by Lemma \ref{ll1.4}, $Ann_R(x) \not\subseteq (P:_RM)$ and so $Ann_R(x)M  \not\subseteq  (P:_RM)M=P$. Thus $V(Ann_R(x)M)\cap V(x)=\emptyset$. On the other
hand, if $P\in Min^p(M) \setminus V(x)$, then for any $b \in Ann_R(x)$, we have $bx=0 \subseteq P$. Since
$x \not \in P$ and $P$ is prime, $bM \subseteq P$. Therefore, $P \in V(Ann_R(x)M)$.  Thus $V(Ann_R(x)M)=Min^p(M) \setminus V(x)$. Both sets $V(Ann_R(x)M)$ and $V(x)$ are closed, and
since they are complementary, they are also open.
\end{proof}

\begin{cor}\label{cc51.11}
Let $M$ be a reduced multiplication $R$-module.  Then
$Min^p(M)$ is a Hausdorff space with a base of open-and-closed sets.
\end{cor}
\begin{proof}
 Given $P \not=P'\in Min^p(M)$, let $x \in P\setminus P'$. Then $V(x)$ and $V(Ann_R(x)M)$ are
disjoint open sets containing $P $ and $P'$, respectively. Hence $Min^p(M)$ is a Hausdorff
space. The family $\{V(x)\}$ is a base for the closed
sets, so that $V(Ann_R(x)M)$ is a base for the open sets. Each member of
the latter base is closed.
\end{proof}

The set of torsion elements
of $M$ with respect to $R$ is $T_0(M) = \{m \in M |rm = 0 \ for \ some\ 0 \not= r \in R\}$.
\begin{cor}\label{c1.151}
Let $M$ be a faithful reduced multiplication $R$-module. Then
$x\in M$ belongs to some minimal prime submodule $P$ of $M$ if and only if $x \in T_0(M)$, i.e., $T_0(M)=\cup_{P \in Min^p(M)}P$.
\end{cor}
\begin{proof}
Let $x \in P$ for some minimal prime submodule $P$ of $M$. Then by Lemma \ref{ll1.4}, $Ann_R(x) \not \subseteq (P:_RM)$. Thus $Ann_R(x) \not=0$ and so
$x \in T_0(M)$. Conversely, consider an element $x$ that belongs to no minimal prime submodule. Then $V(x)=\emptyset$ and so
$V(Ann_R(x)M)=Min^p(M)$  by Theorem \ref{trt1.11}. This implies that $Ann_R(x)M\subseteq \mathfrak{P}_0=0$. Since $M$ is faithful, $Ann_R(x)=0$. Thus $x \not \in T_0(M)$, as desired.
\end{proof}

\begin{prop}\label{p1.11}
Let $M$ be a reduced multiplication $R$-module. Then we have the following.
\begin{itemize}
\item [(a)] If $I$ is an ideal of $R$, then $M/(0:_MI)$ is a reduced $R$-module.
\item [(b)] If $P$ is a prime submodule of $M$ contains $(0:_MI)$ for some ideal $I$ of $R$
and if $P/(0:_MI)$ is a minimal prime submodule of $M/(0:_MI)$. Then $P$ is a minimal prime submodule of $M$.
\end{itemize}
\end{prop}
\begin{proof}
(a) A nilpotent submodule of $M/(0:_MI)$ has the form $K/(0:_MI)$,
where $K$ is a submodule of $M$ and $K^n \subseteq (0:_MI)$ for some positive integer $n$. Then $(IK)^n=0$  and since $M$ has
no non-zero nilpotent submodule, $IK=0$. It follows that $K \subseteq (0:_MI)$, as needed.

(b) First note that as $M$ is a multiplication $R$-module, we have  $M/(0:_MI)$ is a multiplication $R$-module. Also, by part (a),  $M/(0:_MI)$ is a reduced $R$-module. Thus $P/(0:_MI)=O_{P/(0:_MI)}$ by Lemma \ref{ll1.3}. It is enough to we show that $P = O_P$ by Lemma \ref{ll1.3}. Let $x \in P$. Then $x+(0:_MI) \in P/(0:_MI)=O_{P/(0:_MI)}$. This implies that
$$
Ann_R(x) \cap Ann_R((0:_MI)) \not\subseteq (P/(0:_MI):_RM/(0:_MI).
$$
It follows that $Ann_R(x) \not \subseteq (P:_RM)$, as needed.
\end{proof}

\begin{thm}\label{t1.12}
Let $M$ be a reduced multiplication $R$-module. Then for each ideal $I$ of $R$, $(0:_MI)=\mathfrak{P}_{(0:_MI)}$.
\end{thm}
\begin{proof}
Clearly, $(0:_MI)\subseteq \mathfrak{P}_{(0:_MI)}$. For reverse inclusion we show that if $x \not \in (0:_MI)$, then there exists $ P \in Min^p(M)$ such that $x \not \in P$. We have $xI \not=0$ and hence $xa\not=0$ for some $a \in I$. Assert that $xa \not \in (0:_MI)$ since otherwise, $xa^2=0$. It follows that $(Rxa)^2=0$ and so $xa=0$ since $M$ is reduced. Now since by Proposition \ref{p1.11} (a), $M/(0:_MI)$ is a reduced module, there exists $P/(0:_MI) \in Min^p(M/(0:_MI))$ such that $ax+(0:_MI) \not \in P/(0:_MI)$ by Corollary \ref{C1}. Hence $ax \not \in P$ and so $x \not \in P$. By Proposition \ref{p1.11} (b), $P \in Min^p(M)$ as desired.
\end{proof}

\begin{thm}\label{t1.45}
Let $M$ be a reduced multiplication $R$-module. Then we have the following.
\begin{itemize}
\item [(a)] $V(x)=V((0:_MAnn_R(x)))$ for each $x \in M$.
\item [(b)] $(0:_MAnn_R(Jx))=(0:_MAnn_R(JM)) \cap (0:_MAnn_R(x))$ for each ideal $J$ of $R$ and $x \in M$.
\end{itemize}
\end{thm}
\begin{proof}
(a) As $x \in (0:_MAnn_R(x))$, we have $V((0:_MAnn_R(x)))\subseteq V(x)$. Now let
$P$ be a minimal prime submodule of $M$ containing $x$. Then  there exists
 $a \in Ann_R(x) \setminus (P:_RM)$ by Lemma \ref{ll1.4}. Then, for any $y \in (0:_MAnn_R(x))$, we have $ay=0 \in P$. It follows that $y \in P$. Therefore, $(0:_MAnn_R(x))\subseteq P$. Thus $V(x)\subseteq V((0:_MAnn_R(x)))$.

(b) Let $J$ be an ideal of $R$ and $x \in M$. As $Ann_R(x) \subseteq Ann_R(Jx)$ and $Ann_R(JM) \subseteq Ann_R(Jx)$, we have
$$
(0:_MAnn_R(Jx))\subseteq (0:_MAnn_R(JM)) \cap (0:_MAnn_R(x)).
$$
Now let $z \in (0:_MAnn_R(JM)) \cap (0:_MAnn_R(x))$. Assume that $t \in Ann_R(Jx)$. Then $Jt \subseteq Ann_R(x)$ and so $tJz=0$. Since $M$ is a multiplication $R$-module $tzR=IM$ for some ideal $I$ of $R$. Hence, $I \subseteq Ann_R(JM)$ and so $Iz=0$. Thus $t(Rz)^2=0$. It follows that $(Rtz)^2=0$. This implies that $zt=0$ since $M$ is a reduced multiplication $R$-module. Therefore, $zAnn_R(Jx)=0$ and so $z \in (0:_MAnn_R(Jx))$.
\end{proof}

\begin{cor}\label{c1.143}
Let $M$ be a reduced multiplication $R$-module. Then for each $x \in M$, $(0:_MAnn_R(x))=\mathfrak{P}_{Rx}$.
\end{cor}
\begin{proof}
This follows from Theorem \ref{t1.12} and Theorem \ref{t1.45} (a).
\end{proof}

\begin{thm}\label{t99.6}
Let $M$ be a reduced multiplication $R$-module. Then the following are equivalent:
\begin{itemize}
\item [(a)] For each $x,y \in M$, $\mathfrak{P}_x=\mathfrak{P}_y$ and $x\in  N$ imply that $y \in N$;
\item [(b)] For each $x,y \in M$, $V(Rx)=V(Ry)$ and $x\in N$ imply that $y \in N$;
\item [(c)] For each $x,y \in M$, $Ann_R(x)=Ann_R(y)$ and $x\in  N$ imply that $y \in N$.
\end{itemize}
\end{thm}
\begin{proof}
$(a)\Rightarrow (b)$
Let for $x,y \in M$, $V(Rx)=V(Ry)$ and $x \in N$. Then  $\mathfrak{P}_x=\mathfrak{P}_y$. Thus by part (a), $y \in N$.

$(b)\Rightarrow (c)$
Let for $x, y \in M$, $Ann_R(x)=Ann_R(y)$ and $x \in N$. Then $(0:_MAnn_R(x))=(0:_MAnn_R(y))$.  Thus by Theorem \ref{t1.45} (a), $V(Rx)=V(Ry)$. Hence by part (b), $y \in N$.

$(c)\Rightarrow (a)$
Let for $x, y \in M$, $\mathfrak{P}_{Rx}=\mathfrak{P}_{Ry}$ and $x \in N$. Then $(0:_MAnn_R(x))=(0:_MAnn_R(y))$ by Corollary \ref{c1.143}. Hence $$
Ann_R(x)=Ann_R(0:_MAnn_R(x))=Ann_R(0:_MAnn_R(y))=Ann_R(y).
$$
Thus by part (c), $y \in N$.
\end{proof}

\section{$z^\circ$-submodules}
\begin{defn}\label{d9.1}
We say that a proper submodule $N$ of an $R$-module $M$ is a \textit{$z^\circ$-submodule} of $M$ if  $\mathfrak{P}_x \subseteq N$  for all  $x \in N$.
\end{defn}

\begin{rem}\label{r9.1}
Let $M$ be an $R$-module.
If $N$ is a $z^\circ$-submodule of $M$, then for
each $x \in N$ we have  $\mathfrak{P}_x\not=M$, i.e. $x$ contained at least in a minimal prime submodule of $M$.   Clearly, every minimal prime submodule of $M$ is a $z^\circ$-submodule of $M$. Also the family of
$z^\circ$-submodules of $M$ is closed under intersection. Therefore, if  $\mathfrak{P}_0\not=M$, then $\mathfrak{P}_0$ is a $z^\circ$-submodule of $M$ and it is contained in every $z^\circ$-submodule of $M$.
\end{rem}

\begin{prop}\label{pp001.14}
Let $N$ be a $z^\circ$-submodule of a faithful multiplication $R$-module $M$. Then
for each $r \in R$, $(N:_Mr)$ is a $z^\circ$-submodule of $M$.
\end{prop}
\begin{proof}
As $M$ is a faithful multiplication $R$-module, one can see that  $r\mathfrak{P}_{x}\subseteq \mathfrak{P}_{rx}$.
Suppose that $r \in R$ and $x\in (N:_Mr)$. Then $rx \in N$. So by assumption, $\mathfrak{P}_{rx} \subseteq N$. Since, $r\mathfrak{P}_{x}\subseteq \mathfrak{P}_{rx}$, we have $r\mathfrak{P}_{x} \subseteq N$. Thus $\mathfrak{P}_{x} \subseteq (N:_Mr)$.
\end{proof}

\begin{thm}\label{t991.11}
Let $M$ be a reduced multiplication $R$-module, $x , y \in M$, and $y \in Ann_R(x)M$.
If $Ann(x)M = \mathfrak{P}_{y}$, then $x+y \not \in T_0(M)$. The converse holds when $M$ is faithful and $Ann_R(x)M$ is a $z^\circ$-submodule of $M$.
\end{thm}
\begin{proof}
First assume that $Ann(x)M = \mathfrak{P}_{y}$ and assume contrary that $x+y \in P$ for some minimal prime submodule $P$  of
$M$ and seek a contradiction. We consider two cases. First, let $x \in P$, then $y \in P$ implies that $\mathfrak{P}_{y}\subseteq P$, i.e., $Ann(x)M = \mathfrak{P}_{y} \subseteq P$, which is impossible by Theorem \ref{trt1.11}. Now let $x \not\in  P$, then we must have $y \not \in P$, i.e.,
$Ann(x)M=  \mathfrak{P}_{y} \not\subseteq P$, which is again impossible by Theorem \ref{trt1.11}. Conversely, if $x+y \not \in T_0(M)$, then $x+y \not \in P$ for each $P \in Min^p(M)$ by Corollary \ref{c1.151}. Let $P$ be a minima1
prime submodule of $M$ with $y \in P$, then $x+y \not \in P$ implies that $x \not \in P$, i.e., $Ann(x)M \subseteq P$ by Theorem \ref{trt1.11}. Hence $Ann(x)M \subseteq \mathfrak{P}_{y}$. The reverse inclusion follows from the fact that $Ann_R(x)M$ is a $z^\circ$-submodule of $M$.
\end{proof}

\begin{prop}\label{p9.1}
Let $N$ be a proper submodule of an $R$-module $M$. Then $N$ as an $R$-submodule is a $z^\circ$-submodule if and only if as an $R/Ann_R(M)$-submodule is a $z^\circ$-submodule.
\end{prop}
\begin{proof}
This is straightforward.
\end{proof}

\begin{lem}\label{l9.61}
Let $M$ be a reduced multiplication $R$-module. Then the following are equivalent:
\begin{itemize}
\item [(a)] $N$ is a $z^\circ$-submodule of $M$;
\item [(b)] $x \in N$ implies that $(0:_MAnn_R(x)) \subseteq N$;
\item [(c)] $N=\sum_{x\in N} \mathfrak{P}_x$.
\end{itemize}
\end{lem}
\begin{proof}
$(a)\Leftrightarrow (b)$
This follows from Corollary \ref{c1.143}.

$(a)\Leftrightarrow (c)$
This is straightforward.
\end{proof}

\begin{prop}\label{ppp9.1}
Let $M$ be a reduced multiplication $R$-module.
Then we have the following.
\begin{itemize}
\item [(a)] If $N$ is a $z^\circ$-submodule of $M$, then each element of $N$ is torsion, i.e., $N \subseteq T_0(M)$.
\item [(b)] For each $x \in T_0(M)$, $(0:_MAnn_R(x))$ is a $z^\circ$-submodule of $M$ and so every submodule of $T_0(M)$ is contained in a sum of $z^\circ$-submodules of $M$.
\end{itemize}
\end{prop}
\begin{proof}
(a) Let $N$ be a $z^\circ$-submodule of $M$ and $x \in N$. By Lemma \ref{l9.61}, $(0:_MAnn_R(x)) \subseteq N$. This implies that $Ann_R(x)\not=0$ since $N$ is proper. Thus $x \in T_0(M)$.

(b) Since $x \in T_0(M)$, $(0:_MAnn_R(x))\not=M$. Let $ y \in (0:_MAnn_R(x))$. Then $yAnn_R(x)=0$. This implies that $Ann_R(x) \subseteq Ann_R(y)$ and so $(0:_MAnn_R(y)) \subseteq (0:_MAnn_R(x))$. Now let $N$ be a submodule of $T_0(M)$. Then
$$
N=\sum_{x \in N}Rx\subseteq \sum_{x \in N}(0:_MAnn_R(x)).
$$
\end{proof}

An $R$-module $M$ is said to be a \emph{comultiplication module} if for every submodule $N$ of $M$ there exists an ideal $I$ of $R$ such that $N=(0:_MI)$ equivalently, for each submodule $N$ of $M$, we have $N=(0:_MAnn_R(N))$  \cite{AF07}.
\begin{defn}\label{d9.1}
We say that an $R$-module $M$ satisfies \textit{Property $dac$} if for each $x \in M$, $(0:_MAnn_R(x))$ be a cyclic submodule of $M$.
\end{defn}
Clearly every comultiplication $R$-module satisfies the Property $dac$.

\begin{thm}\label{t9.6}
Let $M$ be a reduced multiplication $R$-module which satisfies the Property $dac$. Then the following are equivalent for a proper submodule $N$ of $M$:
\begin{itemize}
\item [(a)] For each $x,y \in M$, $\mathfrak{P}_x=\mathfrak{P}_y$ and $x\in  N$ imply that $y \in N$;
\item [(b)] For each $x,y \in M$, $V(Rx)=V(Ry)$ and $x\in N$ imply that $y \in N$;
\item [(c)] $x \in N$ implies that $(0:_MAnn_R(x)) \subseteq N$;
\item [(d)] For each $x,y \in M$, $Ann_R(x)=Ann_R(y)$ and $x\in  N$ imply that $y \in N$;
\item [(e)] $N$ is a $z^\circ$-submodule of $M$;
\item [(f)] $N=\sum_{x\in N} \mathfrak{P}_x$.
\end{itemize}
\end{thm}
\begin{proof}
$(b)\Rightarrow (c)$
Let $x \in N$. Then  $V(x)=V((0:_MAnn_R(x)))$  by Theorem \ref{t1.45}. As $M$ satisfies the Property $dac$, there exists $y \in M$ such that  $(0:_MAnn_R(x))=Ry$. Thus by part (b), $(0:_MAnn_R(x))=Ry \subseteq N$.

$(c)\Rightarrow (d)$
Let for $x, y \in M$, $Ann_R(x)=Ann_R(y)$ and $x \in N$. Then $(0:_MAnn_R(x))=(0:_MAnn_R(y))$. By part (c), $(0:_MAnn_R(x)) \subseteq N$. Thus $Ry \subseteq (0:_MAnn_R(y))\subseteq N$ and so $y \in N$.

Now the result follows from Theorem \ref{t99.6} and Lemma \ref{l9.61}.
\end{proof}

\begin{thm}\label{t7.67}
Let $M$ be a reduced multiplication $R$-module which satisfies the Property $dac$. Then the Jacobson radical $Rad_N(M)$ of $M$ is zero if and only if every $z^\circ$-submodule of $M$ of is a $z$-submodule of $M$.
\end{thm}
\begin{proof}
$(\Rightarrow)$:  Let $N$ be a $z^\circ$-submodule of $M$. Assume that $\mathcal{M}(x)=\mathcal{M}(y)$ with $x \in N$. It is enough to show that $Ann_R(x)=Ann_R(y)$. To see this, let $a \in Ann_R(x)$ and assume contrary that $a \not \in Ann_R(y)$. Then $0\not=ay\not \in Rad_N(M)$. Thus there exists a maximal submodule $T$ of $M$ such that $a \not \in (T:_RM)$ and $ y \not \in T$. Now $ax=0 \in T$ follows that $x \in T$. Hence, $\mathcal{M}(x)=\mathcal{M}(y)$ implies that $ y \not \in T$, a contradiction.

$(\Leftarrow)$:
Assume contrary that  $0 \not = x \in Rad_N(M)$. Since $M$ is a reduced module there exists a minimal prime submodule $P$ of $M$ not containing $x$. As $\mathcal{M}(x)=\mathcal{M}(0)$ and $P$ is a $z^\circ$-submodule of $M$, we have $x \in P$ by assumption. This is a contradiction.
\end{proof}

\begin{thm}\label{t97.9}
Let $M$ be a reduced multiplication $R$-module which satisfies the Property $dac$ and $N$ be a $z^\circ$-submodule of $M$. Then every
minimal prime submodule over $N$ is a prime $z^\circ$-submodule of $M$. In particular, $rad(N)$ is a  $z^\circ$-submodule of $M$.
\end{thm}
\begin{proof}
Let $P$ be a minimal prime submodule over $N$. Assume that $Ann(x)= Ann(y)$, where
$x \in P$ and $y \in M$. Since $P/N$ is a minimal prime submodule of $M/N$,
By Lemma \ref{ll1.4}, there exists $c \in Ann_R(x+N)) \setminus (P/N:_RM/N)$. Thus $cx \in N$ and $c \not \in (P:_RM)$. Clearly, $Ann(cx)= Ann(cy)$. As $N$ be a $z^\circ$-submodule of $M$, we have
$cy \in N\subseteq P$. As $c \not \in (P:_RM)$ and $P$ is a prime submodule, $y \in P$ as needed. The last assertion is clear.
\end{proof}

\begin{cor}\label{c97.10}
Let $M$ be a reduced multiplication $R$-module which satisfies the Property $dac$. Then we have the following.
\begin{itemize}
\item [(a)]
If $f: M\rightarrow M/N$ is the natural epimorphism, where $N$ is a $z^\circ$-submodule of $M$, then every $z^\circ$-submodule of $M/N$ contracts to a $z^\circ$-submodule of $M$.
\item [(b)]
A submodule $N$ of a $M$ is a $z^\circ$-submodule of $M$ if and only if
it is an intersection of prime $z^\circ$-submodules of $M$.
\item [(c)]
Every maximal $z^\circ$-submodule is a prime $z^\circ$-submodule.
\item [(d)] If $P$ is a prime submodule of $M$, then either $P$ is a $z^\circ$-submodule or contains a maximal  $z^\circ$-submodule which is a prime $z^\circ$-submodule.
\end{itemize}
\end{cor}

\begin{prop}\label{p97.13}
Let $M$ be a reduced multiplication $R$-module which satisfies the Property $dac$. Let $N_i$ for $1\leq i\leq n$ be proper submodule of $M$ such that for each $i\not=j$, $(N_i:_RM)$ and $(N_j:_RM)$ be co-prime ideals of $R$. Then $\bigcap^n_{i=1}N_i$ is a $z^\circ$-submodule of $M$ if and only if each
$N_j$ for $1\leq j\leq n$  is a $z^\circ$-submodule of $M$.
\end{prop}
\begin{proof}
Assume that
$\bigcap^n_{i=1}N_i$ is a $z^\circ$-submodule of $M$ and $1\leq j\leq n$. We show that $N_j$ is a $z^\circ$-submodule of $M$. So let $Ann_R(x)=Ann_R(y)$ for some $x \in N_j$ and $y \in M$. As for each $i\not=j$, $(N_i:_RM)$ and $(N_j:_RM)$ are co-prime ideals of $R$, we have $\bigcap^n_{i=1, i \not=j}(N_i:_RM)$ and $(N_j:_RM)$ are co-prime ideals of $R$. Therefore $1=a+b$ for some $a \in \bigcap^n_{i=1, i \not=j}(N_i:_RM)$ and $b \in (N_j:_RM)$. Thus $y=ay+by$
and $Ann_R(ax)=Ann_R(ay)$. Since $M$ is a multiplication $R$-module, $ax \in \bigcap^n_{i=1}N_i$. Thus $\bigcap^n_{i=1}N_i$ is a $z^\circ$-submodule of $M$ implies that $ay \in \bigcap^n_{i=1}N_i \subseteq N_j$. Now since $by \in N_j$, $y \in N_j$ and we are done. The converse is clear.
\end{proof}

Let $M$ be an $R$-module. The set of zero divisors of $R$
on $M$ is $Zd_R(M)= \{r \in R| rm=0 \ for \ some\ nonzero\ m \in M\}$.
\begin{prop}\label{p09.1}
Let $M$ be a faithful finitely generated multiplication $R$-module. If $N$ is a $z^\circ$-submodule of $ M$. Then $(N:_RM) \subseteq Zd_R(M)$.
\end{prop}
\begin{proof}
Let $r \in (N:_RM)$ and $M=\langle x_1,x_2,...,x_t\rangle$. Then $rx_i \in N$ for $i=1,2,...,t$. Thus $\mathfrak{P}_{rx_i} \subseteq N$ since $N$ is a $z^\circ$-submodule of $M$. This implies that for each $x_i$ there exists a minimal prime submodule $P_i$ of $M$ such that $rx_i \in P$ for $i=1,2,...,t$. As $M$ is a multiplication $R$-module, $Rx_i=I_iM$ for some ideal $I_i$ of $R$. By \cite[Proposition 1.5]{MR2378535}, $(P_i:_RM)$ is a minimal prime ideal of $R$. Thus by \cite[Theorem 84]{MR0345945}, $(P_i:_RM)  \subseteq Zd_R(M)$. Therefore, $rI_i \subseteq Zd_R(M)$ and so $rI_i M\subseteq Zd_R(M)M$ for $i=1,2,...,t$. This in turn implies that $r \in (Zd_R(M)M:_RM)$.
By \cite[Theorem 10]{MR933916},  $(Zd_R(M)M:_RM)=Zd_R(M)$. Thus $r \in Zd_R(M)$. Therefore, $(N:_RM) \subseteq Zd_R(M)$.
\end{proof}

An $R$-module $M$ satisfies \textit{Property $\mathcal{T}$} if for every
finitely generated submodule $N$ of $M$ with $N \subseteq T_0(M)$, there exists a nonzero  $r \in R$ with $rN = 0$, or equivalently
$Ann_R(N) \not=0$ \cite{MR3661610}.

\begin{thm}\label{t97.10}
Let $M$ be a reduced multiplication $R$-module satisfying Property $\mathcal{T}$ and $N$ be a submodule of $T_0(M)\not=M$, then $N$ is contained in a $z^\circ$-submodule.
\end{thm}
\begin{proof}
We define $N_0= N$, $N_1 =\sum_{x \in N_0}(0:_MAnn_R(x))$ and if $\alpha$ is a limit ordinal
we define $N_{\alpha} = \cup_{\beta <\alpha}N_{\beta}$, where $\beta$ is an ordinal, and finally if $\alpha= \beta + 1$,
we define $N_{\alpha} = \sum_{x \in N_\beta}(0:_MAnn_R(x))$. Then we have the following ascending chain
$$
N_0\subseteq N_1\subseteq ...\subseteq N_{\alpha}\subseteq N_{\alpha+1}\subseteq ....
$$
Since $M$ is a set, there exists the smallest
ordinal $\alpha$ such that $N_{\alpha}=N_{\gamma}$, for all $\gamma \geq \alpha$.  We claim that $N_{\alpha}$ is a proper submodule
which is also a $z^\circ$-submodule. If $N_{\alpha}$ is a proper submodule, it is certainly a  $z^\circ$-submodule,
for $N_{\alpha}=N_{\alpha+1}=\sum_{x \in N_{\alpha}}(0:_MAnn_R(x))$. This implies that $(0:_MAnn_R(x))\subseteq N_{\alpha}$ for each $x \in N_{\alpha}$, and therefore by Lemma \ref{l9.61}, we are through. Thus it
remains to be shown that $N_{\alpha}$, is a proper submodule. To see this, it suffices to
show that $N_{\alpha}\subseteq T_0(M)$ each $\alpha$. We proceed by
transfinite induction on $\alpha$. For $\alpha=0$, it is evident. Let us assume it true
for all ordinals $\beta < \alpha$ and prove it for $\alpha$. If $\alpha$ is a limit ordinal, then
 $N_{\alpha} = \cup_{\beta <\alpha}N_{\beta}$ and therefore $N_{\alpha}\subseteq T_0(M)$. Now let $\alpha = \beta + 1$
be a non limit ordinal, then $N_{\alpha} = \sum_{x \in N_\beta}(0:_MAnn_R(x))$. We must show that
each element $y$ of $N_{\alpha}$, is a torsion. Put $y = y_l + y_2 + . . . + y_n$, where
$y_i \in (0:_MAnn_R(x_i))$, $x_i \in N_{\beta}$, $i = 1,2, . . . , n$. But by the induction hypothesis $N_{\beta}\subseteq T_0(M)$. Now by the Property $\mathcal{T}$, there exists $0\not=b \in Ann_R(Ry_1+Ry_2+ . . +Ry_n)$, i.e., $by = 0$.
\end{proof}

\begin{cor}\label{c97.11}
If $M$ is a reduced multiplication $R$-module satisfying Property $\mathcal{T}$, then every
maximal submodule of $M$ that contained in $T_0(M)\not=M$ is a $z^\circ$-submodule.
\end{cor}

\begin{cor}\label{c97.12}
If $M$ is a reduced multiplication $R$-module satisfying Property $\mathcal{T}$, and $N \subseteq T_0(M)\not=M$ is a submodule of $M$, then there is the smallest $z^\circ$-submodule containing
$N$ and also there is a maximal $z^\circ$-submodule containing $N$ which is also a prime
$z^\circ$-submodule.
\end{cor}
\bibliographystyle{amsplain}

\end{document}